\documentclass[11pt,reqno]{amsart} 

\usepackage[margin = 1.1in]{geometry}
\usepackage{amsfonts, amsmath,amscd, amssymb, cite, color, latexsym, mathrsfs, mathtools,
slashed, stmaryrd, verbatim, wasysym }
\usepackage[all]{xy}
\usepackage[pdftex]{graphicx}
\usepackage{hyperref}
\usepackage[applemac]{inputenc}
\usepackage[titletoc,title]{appendix}
\usepackage{tikz-cd}
\usepackage{oubraces}
\usepackage{cancel}
\usepackage{float}
\allowdisplaybreaks
\usepackage{tensor}

\usepackage{import}
\usepackage{xifthen}
\usepackage{pdfpages}
\usepackage{transparent}

\DeclareMathAlphabet{\mathboondoxfrak}{U}{BOONDOX-frak}{m}{n}

\usepackage{tikz}
\usetikzlibrary{decorations.pathmorphing, matrix, calc, arrows}

\newlength{\LETTERheight}
\AtBeginDocument{\settoheight{\LETTERheight}{I}}

\newtheorem{theorem}{Theorem}
\newtheorem{corollary}[theorem]{Corollary}

\newtheorem{lemma}[theorem]{Lemma}
\newtheorem{proposition}[theorem]{Proposition}

\numberwithin{equation}{section}
\numberwithin{theorem}{section}

\let\oldsqrt\sqrt
\def\sqrt{\mathpalette\DHLhksqrt}
\def\DHLhksqrt#1#2{%
\setbox0=\hbox{$#1\oldsqrt{#2\,}$}\dimen0=\ht0
\advance\dimen0-0.2\ht0
\setbox2=\hbox{\vrule height\ht0 depth -\dimen0}%
{\box0\lower0.4pt\box2}}

\newcommand{\wt}[1]{\widetilde{#1}}

\newcommand{\Rp}{\mathbb{R}^+}

\newcommand\pa{\partial}

\newcommand\eps\varepsilon
\renewcommand\epsilon\varepsilon

\newcommand{\Cl}{\mathbb{C}l}

\newcommand\CI{{\mathcal{C}}^{\infty}}

\newcommand\End{\operatorname{End}}

\newcommand\Id{\operatorname{Id}}

\newcommand\inj{\operatorname{inj}}

\newcommand\Ker{\operatorname{Ker}}

\newcommand\supp{\operatorname{supp}}

\newcommand\Vol{\operatorname{Vol}}

\newcommand\paperintro%
        {%
         }
\newcommand\paperbody%
        {%
         }

\newcommand\bbR{\mathbb{R}}

\newcommand\bbZ{\mathbb{Z}}

\newcommand\cB{\mathcal{B}}
\newcommand\cC{\mathcal{C}}
\newcommand\cD{\mathcal{D}}

\newcommand\cF{\mathcal{F}}

\newcommand\cI{\mathcal{I}}

\newcommand\cO{\mathcal{O}}

\newcommand\cW{\mathcal{W}}

\newcommand\sD{\mathscr{D}}

\DeclareMathAlphabet{\mathpzc}{OT1}{pzc}{m}{it}
\newcommand{\cl}{\mathpzc{cl}}

\newcommand{\D}{\slashed{D}}

\begin{document}
\title[The Calder\'{o}n problem for the Fractional Dirac Operator]{The Calder\'{o}n problem for the Fractional Dirac Operator}
\author{Hadrian Quan, Gunther Uhlmann}
\maketitle

\begin{abstract}
    We show that knowledge of the source-to-solution map for the fractional Dirac operator acting  over sections of a Hermitian vector bundle over a smooth closed connencted Riemannian manifold of dimension $m\geq 2$ determines uniquely the smooth structure, Riemannian metric, Hermitian bundle and connection, and its Clifford modulo up to a isometry. We also mention several potential applications in physics and other fields.
\end{abstract}

\section{Introduction} 

The Calder\'on problem asks whether one can determine the electrical conductivity of a medium by making voltage and current measurements at the boundary.  In the anisotropic case, that is when the conductivity depends on direction. It is modelled by a positive-definite symmetric matrix. It was shown in \cite{L-U} that this problem is equivalent to determining a Riemannian metric from the associated Dirichlet to Neumann (DN) map associated with harmonic functions. Therefore this problem can be considered on a compact Riemannian manifold with boundary and arises also in the AdS/CFT correspondence \cite{CG}. See \cite{U1}, \cite{U2} for more details and other results.

The study of the fractional Calder\'{o}n problem was initiated in \cite{GSU} where the unknown potential in the fractional Schr\"{o}dinger equation on a bounded domain in the Euclidean space was determined from exterior
measurements. An important generalization was in the work \cite{GLX} where the Euclidean Laplacian is replaced by the Laplace-Beltrami operator of a Riemannian metric. Following these works, inverse
problems of recovering lower order terms for fractional elliptic equations have been studied extensively, see for example \cite{GRSU}, \cite{GLX}, \cite{RS1}, \cite{RS2}, \cite{BGU}, \cite{CLR}, \cite{Covi1}, \cite{Covi2}, \cite{Covi3}, \cite{CG-FR}, \cite{CMRU},  \cite{Li1}, \cite{Li2}, \cite{Ru} for some of the  contributions. In all of those papers, it is assumed that the leading order coefficients are known.

In the article \cite{FGKU}, it was solved the fractional anisotropic Calder\'{o}n problem on closed Riemannian manifolds of dimensions two and higher that the knowledge of the local source-to-solution map
for the fractional Laplacian, given an arbitrary small open nonempty a priori known subsef a smooth closed connected Riemannian manifold, determines the Riemannian manifold up to an isometry. This can be viewed as a nonlocal analog of the anisotropic Calder\'{o}n problem in the setting of closed Riemannian manifolds, which is  open in dimensions three and higher.
\vspace{.4cm}

We consider in this paper the anisotropic Calderon problem for the fractional Dirac operator acting on sections of a Hermitian vector bundle over a smooth closed connected Riemannian manifold $M$ of dimension $m\geq 2$. Just from local measurements confined to an arbitrary non-empty open set, we show that the fractional Dirac operator determines the smooth structure, Riemannian metric, Hermitian bundle and connection, and its Clifford module, up to a isometry which fixes the set in question. We will first briefly motivate the concept of a generalized Dirac operator, and its fractional powers, before explaining the types of measurements we allow, and stating our results.

In his attempts to quantize electromagnetism, Dirac defined the Dirac operator while trying to find a 1st order differential operator such that its square was the ``Laplacian" on Minkowski space (in fact this was a factorization of the d'Alembertian). The data of the Clifford algebra (see \S \ref{dirac-def}) entered into his definition after he realized this was not possible for a scalar operator. On a vector bundle $E$ over a Riemannian manifold $M$, the definition of a generalized Dirac operator $\D$ is motivated by the desire to capture this property that $\D\circ \D=\Delta_E$ is a generalized Laplace-type operator on $E$, i.e. the Bochner Laplacian $\nabla^*\nabla$ of some connection up to lower order terms. This factorization of $\Delta_E$ implies $\D$ is elliptic and hence admits a spectral resolution  $\{\varphi_k,\lambda_k\}$ of $L^2(M;E)$-orthonormal eigensections, with discrete eigenvalues accumulating only at $\pm \infty$. From the spectral theorem we can define for $\alpha\in (0,1)$, the fractional Dirac operator by $\D^\alpha \varphi_k = \lambda_k^\alpha \varphi_k$. From unique solvability of the ``Poisson equation", 
\[ \D^\alpha u=f, \] 
for data $f$ orthogonal to $\Ker(\D)$, we can consider the operator which maps a source function $f\in \CI_0(\cO; E|_\cO)\cap \Ker(\D)^\perp$ to the solution,
\[ L_{\D,\cO}: f\mapsto u:=(\D^{-\alpha}f)|_{\cO} , \]
for $f\perp \Ker(\D)$. We refer to this operator as the \emph{source-to-solution map}. We prove,

\begin{theorem}\label{main-thm}

Let $\sD_j:=(M_j, E_j, g_j, h_{E_j}, \nabla^{E_j})$ for $j=1,2$ be two Dirac bundles (see \S \ref{dirac-def}) over closed Riemannian manifolds of dimension $n\geq 2$, and let $\cO_j\subset M_j$ be non-empty open sets. Assuming there exists a diffeomorphism $\psi:\cO_1\to\cO_2$ satisfying 
\[ L_{\D_1,\cO_1}(\psi^*f) = \psi^*L_{\D_2,\cO_2}(f) ,  \]
for all $f\in \CI_0(\cO_1;E_1|_{\cO_1})$ orthogonal to $\Ker(\D_1)$. Then there is an isomorphism of Hermitian vector bundles $\Psi:E_1\to E_2$, covering an isometry $\psi:M_1\to M_2$ which restricts to $\Psi|_{E_1|_{\cO_1}}=\psi$.
\end{theorem}

\begin{corollary}\label{cor}

Let $\sD:=(M, E, g, h_{E}, \nabla^{E})$ be a Dirac bundle over a smooth closed manifold of dimension $m\geq 2$, with $\cO\subset M$ a non-empty open subset. Let $\{\varphi_k\}_{k= 1}^\infty\subset L^2(M; E)$ be the collection of positive eigensections with corresponding eigenvalues $\{\lambda_k\}_{k=1}^\infty\subset \Rp$. Then the partial spectral data plus clifford multiplication $\cl: \Cl(TM,g)\to \End(E)$,
\[ (\cO, \varphi_k|_{\cO}, \lambda_k, \cl|_{T^*\cO})  \]
determines the metric, Hermitian vector bundle, and connection, up to an isometry.
\end{corollary}

\subsection{Applications}
Already there has been some interest in studying generalizations of classical equations of physics with respect to a fractional time derivative, (see for example \cite{Laskin}, \cite{JR}).

In this work we consider instead fractional differential operators, corresponding to fractional spatial derivatives. One place where there has been interest in applying such nonlocal operators is in the study of particle physics beyond the standard model, i.e. fields governed by fractional wave equations. The work of \cite{Zavada} considered $n$th powers for $n\geq 2$ of the d'Alembert operator and demonstrated that for $n>2$ the covariant wave equations generated by $\Box_g^{1/n}$ generate a representation of $SU(n)$. Along similar lines \cite{HERRMANN} developed a form of local gauge invariance for such fractional fields and used this to deduce the Baryon mass spectrum via a fractional extension of the classical Zeeman effect \cite{Zeeman}. For a more recent development of such fractional field theories see the work of \cite{Cuprate} who suggest that anomalous power laws for the ``strange metal" properties of Cuprate can be explained if the metal interacts with light via a gauge theory of fractional dimension; also the work of \cite{LnLP} who introduce a theory of fractional electromagnetism, which is motivated in part by a generalization of the Caffarelli-Silvestre extension to the case of the Hodge Laplacian. 

The most surprising applications relate such nonlocal operators to questions coming from quantum information theory. A crucial first step in studying entanglement properties of algebraic quantum field theories is the Reeh-Schlieder theorem \cite{Witten}, which states all local fields in the field algebra of spacetime are entangled with fields localized to all other regions (c.f. \cite{Strohmaier} for a more mathematical exposition of this theorem in the language of von Neumann algebras). The work of \cite{Verch} first gave a proof of the Reeh-Schlieder theorem for static spacetimes using only the strong anti-locality property of $\Delta_g^{1/2}$. In the work of \cite{GSU} they demonstrate the comparable anti-locality property for $\Delta_g^\alpha$ for $\alpha\in (0,1)$, using standard techniques in inverse problems (Carleman estimates, etc.). That this entanglement property of the field algebra is equivalent to a strong unique continuation principle for certain nonlocal operators suggests other interesting connections between questions in inverse problems and quantum information theory.

\section{Background on Generalized Dirac operators}\label{dirac-def}

Let $(M^m,E^n,h_E,g)$ be a Hermitian vector bundle of rank $n$ over smooth compact Riemannian manifold without boundary. We further call this Hermitian bundle a \emph{Clifford module} if $E\to M$ is bundle over $M$ equipped with bundle morphism, known as \emph{Clifford multiplication},
\[ \cl:\Cl(T^*M, g)\to \End(E)  \]
from the unique Clifford module on $T^*M$ induced by $g$ to $\End(E)$. As a vector bundle, $\Cl(T^*M,g)$ is isomorphic to $\Lambda^\bullet T^*M$, and as a module has an algebra operation determined by the relation 
\[\cl(\alpha)\cl(\beta)+\cl(\beta)\cl(\alpha)=-2g(\alpha,\beta)\Id \quad \forall \alpha,\beta\in \Lambda^\bullet T^*M. \]
We say the Hermitian metric is compatible with Clifford multiplication if 
\[ h_E(\cl(\theta)v,w)+h_E(v,\cl(\theta)w)=0, \quad \forall v,w\in \CI(M;E) \]
and a choice of Hermitian connection $\nabla^E$ on $E$ is is compatible with Clifford multiplication if 
\[ [\nabla^E,\cl(\theta)] = \cl(\nabla^g\theta) \]
where $\nabla^g$ is the Levi-Civita connection of $g$.

Given a Hermitian bundle with connection both compatible with Clifford multiplication we can construct a \emph{generalized Dirac operator} $\D$, defined by
\[ \D: \CI(M; E)\xrightarrow{\nabla^E} \CI(M;T^*M\otimes E)\xrightarrow{i\cl(-)} \CI(M;E),  \]
and call the collected data $(M, E, g, h_E, \nabla^E)$ a \emph{Dirac bundle}. The generalized Dirac operator is a Dirac operator in the usual sense by arising as the `square root' of a generalized Laplace operator on $E$; compatibility of the Hermitian metric and connection implies 
\[ [\D,f]=[i\cl\circ\nabla^E,f]=i\cl(f),\] 
thus $\D$ has principal symbol
\[ \sigma_1(\D)(x,\xi)=i\cl(\xi), \implies \sigma_2(\D^2)(x,\xi)=|\xi|_g^2 \Id_E, \]
hence $\D^2:=\Delta_E$ is a principally scalar multiple of the metric $g$. Further, we see from this computation that both $\D$ and $\D^2$ are principally scalar elliptic differential operators (of orders 1 and 2 respectively).

\section{Fractional Dirac operator and determination of the Heat kernel}

In this section we give two equivalent definitions of the fractional Dirac operator and use one to define the source-to-solution operator associated to an a priori known open set $\cO\subset M$. Then, following \cite{FGKU}, we show that knowledge of this source-to-solution operator determines the Heat kernel on $\cO$.

From the symbol calculation above we see that $\D$ is a symmetric operator. Because $M$ is a closed manifold, $\D:\CI(M;E)\to \CI(M;E)$ is essentially self-adjoint on its core domain of smooth sections with a self-adjoint extension to $H^1(M;E)$. Unlike $\D^2$, $\D$ fails to be a non-negative operator. On the other hand, its discrete spectrum (excluding the zero eigenvalue) is in correspondence with the spectrum of $\D^2$; the discrete eigenvalues of $\D$ come in positive and negative pairs $\{\pm\lambda_k\}$ (corresponding to an eigenvalue $\lambda_k^2$ of $\D^2$) which we index by their absolute values
\[ 0=\lambda_0<|\lambda_1|\leq |\lambda_2|\leq \cdots \nearrow +\infty \]
for the distinct eigenvalues of $\D$, and denote $d_k$ the multiplicity of $\lambda_k$. Let $\{\varphi_{k_j}\}_{j=1}^{d_k}$ be an $L^2(M;E)$-orthonormal basis for the eigenspace $\Ker(\D-\lambda_k)$ corresponding to $\lambda_k$, and denote $\pi_k:L^2(M; E)\to \Ker(\D-\lambda_k)$ for the orthogonal projection onto the corresponding eigenspace, written as
\[ \pi_k(f) = \sum_{j=1}^{d_k} \langle f,\varphi_{k_j}\rangle_{L^2(M;E)}\varphi_{k_j},  \]
for all $k=0,1,\ldots$. Here $\langle \cdot,\cdot\rangle_{L^2(M;E)}$ is the $L^2$-inner product on sections of $E$ induced by our choice of $h_E$.

Fix $\alpha\in (0,1)$. Given this spectral resolution of $\D$ we can define the Fractional Dirac operator $\D^\alpha:\CI(M;E)\to \CI(M;E)$
\[ \D^\alpha f = \sum_{k\in \bbZ} \lambda_k^\alpha \pi_k(f) , \]
which extends to an unbounded self-adjoint operator on $L^2(M;E)$ with domain $\cD(\D^\alpha)=H^\alpha(M;E)$.

Unlike the scalar case, the nullspace of $\D$ may include more than just the constant functions. Say that dim $\Ker(\D)=d_0$ has orthonormal basis $\{\varphi_j^0\}_{j=1}^{d_0}$, then we can solve the equation
\begin{equation}\label{directD}
    \D^\alpha u = f,  
\end{equation}  
for $f\in \CI_0(\cO; E)$, with $\cO\subset M$ open, whenever we impose
\[ \langle f,\varphi_1^0\rangle_{L^2}=\cdots =\langle f, \varphi_{d_0}^0\rangle_{L^2}=0\]
for a unique solution $u=u^f\in \CI(M; E)$ defined by the condition that $u^f\perp_{h_E} \Ker(\D)$. Associated to the equation \eqref{directD} we can define the local source-to-source solution map $L_{M,g,\cO,\ldots}$ by
\[ L_{\D,\cO}(f):=u^f|_{\cO}=(\D^{-\alpha}f)|_{\cO} \]

We can give an equivalent (spectral theoretic) definition for the Fractional Dirac operator via the heat kernel of the square of the Dirac operator. First note that we can define from the integral formula for the Gamma function,
\[ \Gamma(\alpha) = \int_0^\infty e^{-t}t^{\alpha-1}dt, \implies
    \mu^{-\alpha} = \frac{1}{\Gamma(\alpha)}\int_0^\infty e^{-t\mu} t^{\alpha-1}dt, 
\]
thus we write
\[ (\Delta_E)^{-\alpha} = \frac{1}{\Gamma(\alpha)}\int_0^\infty e^{-t\Delta_E} t^{\alpha-1}dt. \]
Using this formula, we have $\D^\alpha = \D\circ \Delta_E^{\frac{\alpha-1}{2}}=\Delta_E^{\frac{\alpha-1}{2}}\D$, i.e. 
\[ \D^\alpha u = \frac{1}{\Gamma(\frac{1-\alpha}{2})}\int_0^\infty t^{\frac{1-\alpha}{2}} e^{-t\Delta_E}(\D u)\frac{dt}{t}   \]
using the fact that $\D e^{-t\Delta_E}=e^{-t\Delta_E}\D$ (from uniqueness of the solution to the heat equation). 

Using this definition of $\D^{-\alpha}$ we can easily extend the proof of \cite[Thm 1.5]{FGKU} to our setting: 
\begin{theorem}
Let $\sD_j:=(M_j, E_j, g_j, h_{E_j}, \nabla^{E_j})$ for $j=1,2$ be two Dirac bundles over closed Riemannian manifolds of dimension $n\geq 2$, then we denote by $\D_j$ the generalized Dirac operator associated to $\sD_j$. Let $\cO_j\subset M_j$ be non-empty open sets and assume that
\begin{equation}\label{opens-equal}
    (\cO_j, g_j|_{\cO_j}):=(\cO, g)  
\end{equation} 
and that there exists hermitian bundle isomorphism $\phi: E_1|_{\cO_1}\to E_2|_{\cO_2}$. Assume furthermore that
\begin{equation}\label{stsm-equal}
    L_{\D_1,\cO_1}(f) = L_{\D_2,\cO_2}(f)
\end{equation} 
for all $f\in \CI_0(\cO_1, E_1|_{\cO})$ such that $f\perp_{L^2(\cO;E_1|_\cO)} \Ker(\D_1)$, for any smooth extension of $f$ to $M_2$. Then 
\[ e^{-t\Delta_{E_1}}(x,y) = e^{-t\Delta_{E_2}}(x,y), \quad x,y\in \cO,\; t>0  \]
\end{theorem}

\textbf{Remark}: The requirement for well-posed of the source-to-source solution maps that $f$ be orthogonal to the space Harmonic sections is subtle: the local bundle isomorphism $\phi$ is already sufficient to ensure that $\text{dim}\;\Ker(\D_1)=\text{dim}\Ker(\D_2)$. If $f\in \Ker(\D_2)$ then its restriction to $\cO_2$ pullbacks to a section of $E_1|_{\cO_1}$, and in particular is also an element of $\Ker(\D_1)$ as the Dirac operator commutes with pullback. Similarly pullback by the inverse proves the reverse inclusion. The possibility that a non-trivial Harmonic section vanishes on the open sets in consideration is ruled out by the unique continuation principle for the Dirac operator: if $\D f=0$ and $f|_\cO=0$ then $f=0$ (see e.g. \cite[Ch. 8]{BBW}).

\begin{proof}
Choosing $\omega_1\Subset \cO$ non-empty and open, with $\omega_2\subset \cO$ also non-empty open such that $\overline\omega_1\cap \overline\omega_2=\emptyset$. For $f\in \CI_0(\cO; E|_{\cO})$, due to \eqref{opens-equal}, we have for all $m=1,2,\ldots$,
\[ \D_1^{2m-1}f=\D_2^{2m-1}f=\D^{2m-1}f \quad \text{on $\omega_1$}, \]
and $\D^{2m-1}f$ is orthogonal to $\Ker(\D_j)$ for $j=1,2$. Further from \eqref{stsm-equal}, for all $m=1,2,3,\ldots$ that
\[ (\D_1^{-\alpha} \D^{2m-1}f)|_{\cO}=(D_2^{-\alpha} \D^{2m-1}f)|_{\cO}, \]
hence 
\begin{align*}
    \int_0^\infty t^{\frac{\alpha-1}{2}}((e^{-t\Delta_{E_1}}\D_1 - & e^{-t\Delta_{E_2}}\D_2 )\D^{2m-1}f)(x) \frac{dt}{t} = \int_0^\infty t^{\frac{\alpha-1}{2}}((e^{-t\Delta_{E_1}} - e^{-t\Delta_{E_2}} )\Delta^{m}f)(x) \frac{dt}{t} =0 \\
    & =\int_0^\infty t^{\frac{\alpha-1}{2}} \pa_t^m ((e^{-t\Delta_{E_1}} - e^{-t\Delta_{E_2}})f)(x) \frac{dt}{t} 
\end{align*} 
for $x\in \cO,\; m=1,2,\ldots$, where we have again used \eqref{opens-equal} and that $\D^2=\Delta_E$. We next aim to integrate by parts $m$-times, and observe that the boundary terms vanish due to a combination of classical Heat kernel estimates (both short and long time). The relevant Heat kernel estimates which hold for the Heat kernel of the Dirac Laplacian $\Delta_{E_j}$, are
\[ |e^{-t\Delta_{E_j}}(x,y)|\leq Ct^{-(m+1)}e^{-\frac{cdg_{j}(x,y)^2}{t}}, \quad 0<t<1,\; x,y\in M_j  \]
(has a slightly worse exponent in $t$ than what is known for the scalar case, but still suffices for our purposes), and 
\[ ||e^{-t\Delta_{E_j}}||_{L^1\to L^\infty}\leq C t^{-m/2}, \quad t>0 \]
which is precisely the classical estimate of \cite{Varopoulos}, equivalent to the scalar case. The first estimate can be found in \cite[Thm 3.5]{Ludewig}. Combining these two estimates, we can bound the integrand above, $|\pa_t((e^{-t\Delta_{E_1}} - e^{-t\Delta_{E_2}})f)|(x)$, by functions vanishing at the endpoints of the integral. Proceeding to integrate by parts $m$-times we have that 
\[ \int_0^\infty t^{\frac{\alpha-1}{2}-(m+1)}((e^{-t\Delta_{E_1}}-e^{-t\Delta_{E_2}})f)(x)dt=0  \]
and from here we can conclude that $\cF(\chi_{[0,\infty)}\varphi)(s)$ is holomorphic with all derivatives vanishing at $s=0$, for 
\[ \varphi(s)=\frac{((e^{s\Delta_{E_1}}-e^{s\Delta_{E_2}})f)(x)}{s^\alpha},  \]
thus $((e^{-t\Delta_{E_1}}-e^{-t\Delta_{E_2}})f)(x)=0$ for all $x\in \omega_2$ and all $t>0$. By unique continuation of the Heat equation this implies the same equality holds on all of $\cO$. Using that $f\in \CI_0(\omega_1)$ was arbitrary we have
\[ e^{-t\Delta_{E_1}}(x,y)=e^{-t\Delta_{E_2}}(x,y), \quad x,y\in \cO, t>0 \]
as claimed.
\end{proof}

\section{Reconstruction via the Wave equation}

Have shown that the source-to-source solution operators determine the Heat kernel on the open neighborhood $\cO$. Using the Kannai transmutation formula
\[ e^{-t\Delta_{E_j}}(x,y) = \frac{1}{(4\pi t)^{1/2}t}\int_0^\infty e^{-\frac{\tau}{4t}} \frac{\sin(\sqrt{\tau}\sqrt{\Delta_{E_j}})}{\sqrt{\Delta_{E_j}}} d\tau \quad t>0 \]
the local determination of $e^{-t\Delta_{E_j}}(x,y)$ and this equality imply 
\begin{equation}\label{sinc-equal}
     \left(\frac{\sin(t\sqrt{\Delta_{E_1}})}{\sqrt{\Delta_{E_1}}}f\right)(x) = \left(\frac{\sin(t\sqrt{\Delta_{E_2}})}{\sqrt{\Delta_{E_2}}}f\right)(x) ,  
\end{equation} 
for all $t>0$ and $x\in \cO$. 

The benefit of this last equality is that the operator in \eqref{sinc-equal} lets us represent the solution of the initial value problem,
\begin{equation}\label{direct-wave-eqn}
    \begin{cases} (\pa_t^2 - \Delta_{E_j})u_j(t,x)=f(t,x), & (t,x)\in (0,\infty)\times M_j \\ u_j(0,x)=0 , \pa_tu_j(0,x)=0 & x\in M_j \\ \end{cases}
\end{equation}   
for data $f\in \CI_0(\cO\times M_j; E_j)$ via the formula
\[ u_j(t,x) = \int_0^t \frac{\sin(t-s)\sqrt{\Delta_{E_j}}}{\sqrt{\Delta_{E_j}}} f ds.  \]
Associated to \eqref{direct-wave-eqn}, we define the source-to-source solution map for the wave operator on $\cO$ by 
\[ L_{\sD_j,\cO_j}^\text{wave}: \CI_0((0,\infty)\times \cO)\to \CI_0([0,\infty)\times \cO), \quad L_{\sD_j, \cO_j}^\text{wave}(f)=u_j|_{\cO_j}  \]
where $u_j(t,x)$ is the unique solution to \eqref{direct-wave-eqn}. And by the equality in \eqref{sinc-equal} we have equality of the source-to-solution maps for the wave operators for $j=1,2$,
\[ L_{\sD_1,\cO_1}^{\text{wave}}(f) = L_{\sD_2,\cO_2}^{\text{wave}}(f) .  \]
Now we can move to the statement of our main theorem.

\begin{theorem}

Let $\sD_j:=(M_j, E_j, g_j, h_{E_j}, \nabla^{E_j})$ for $j=1,2$ be two Dirac bundles over closed Riemannian manifolds of dimension $n\geq 2$, and let $\cO_j\subset M_j$ be non-empty open sets. Assuming there exists a diffeomorphism $\psi:\cO_1\to\cO_2$ satisfying 
\[ L_{\sD_1,\cO_1}^{\text{wave}}(\psi^*f) =\psi^*(L_{\sD_2,\cO_2}^{\text{wave}}f), \quad \forall f\in \CI_0((0,\infty)\times \cO_2;E).  \]
Then there is an isomorphism of Hermitian vector bundles $\Psi:E_1\to E_2$, covering an isometry $\psi:M_1\to M_2$ which coincides with $\Psi|_{E_1|_{\cO_1}}=\psi$.
\end{theorem}

We follow the proof of \cite{HLOS}, (which itself was an application to closed manifolds of the boundary control method of Belishev \cite{Bel} combined with the crucial unique continuation method of Tataru \cite{Tataru}) extended to the bundle case. The general structure of the proof is as follows 
\begin{enumerate}
    \item The source-to-source solution operator $(\cO,L_{\sD,\cO}^{\text{wave}})$ determines the distance function $d_g$ on $\cO\times \cO$
    \item The \emph{source-to-source solution data} $(\cO, g|_\cO, h_E|_{\cO}, L_{\sD, \cO}^\text{wave})$ determines the \emph{distance data} $(\cO, g|_{\cO}, R(M))$ 
    \item The \emph{distance data} $(\cO, g|_{\cO}, R(M))$ determines the topology, smooth structure and Riemannian structure of $(M,g)$
    \item {\color{red}New:} The distance data determines the isomorphism class of $E\to M$ and Hermitian metric $h_E$
    \item {\color{red}New:} The distance data determines the Hermitian connection $\nabla^E$, and thus the homomorphism of clifford multiplication from the identity $[\nabla^E,\cl(\theta)]=\cl(\nabla^g\theta)$
\end{enumerate}

This uses several facts about the solution to the linear wave equation on sections of $E$: the first is its finite speed of propagation.

\begin{theorem}

Let $T>0$, and $p\in M$ be open and define the open cone
\[ C_{T,p}:=\{(t,x)\in (0,T)\times M: d_g(x,p)< T-t \} .  \]
Let $f\in L^2(\bbR\times M; E)$ and suppose $u$ solves 
\[ \begin{cases}
(\pa_t^2-\Delta_E)u = f & (0,\infty)\times M \\
f|_{C_{T,p}}= 0 \\
u|_{\{t=0\}\times B_T(p)}=\pa_tu|_{\{t=0\}\times B_T(p)} =0
\end{cases}   \]
then $u|_{C_{T,p}}=0$.
\end{theorem}
To give the statement of the relevant unique continuation principle   we also define
\[ M(T,\cO)=\{x\in M: d_g(x,\cO)\leq T\}  \]
for the domain of dependence of the wave equation. As a result of the Carleman estimates established in \cite{EINT}, the proof of the unique continuation theorem is as in the scalar case since the wave equation being considered is principally scalar (i.e. $\pa_t^2 - \Delta_E$) see section 2.5 of \cite{KKI},

\begin{theorem}\label{uniqcont}

Let $T>0$, and $\cO\subset M$ be open and bounded. Let $u\in \CI_0(\bbR\times M; E|_{\cO})$. Let 
\[ \begin{cases}
(\pa_t^2-\Delta_E)u=0 & (0,2T)\times M(T,\cO) \\ u|_{(0,2T)\times \cO}\equiv 0
\end{cases}\] 
Then $u|_{C(T,\cO)}\equiv 0$, for 
\[ C(T,\cO) = \{(t,x)\in (0,2T)\times N : d_g(x,\cO)\leq \min\{t,2T-t\}\} \]
\end{theorem}

Given the proper form of unique continuation we should be able to prove density of solutions with sources from the set,
\[ \cF(T,\cO) = \{ f\in \CI_0(\bbR\times M; E) | \supp(f)\subset (0,T)\times \cO \}  \]

\begin{theorem}[Approximate controllability]\label{approx-contr}
Let $\cO\subset M$ be open and bounded. For $T>0$ the set
\[ \cW_T=\{ Wf(T,\cdot): f\in \cF(T,\cO) \}   \]
is dense in $L^2(M(T,\cO);E)$. Further, by considering time reparametrization, we obtain that $\{Wf(T,\cdot): f\in \CI_0((T-r,T)\times \cO; E)$ is dense in $L^2(M(r,\cO);E)$ for all $r>0$.
\end{theorem}
\begin{proof}

From the finite speed of propagation we have that $\cW_T\subset L^2(M(T,\cO);E)$. Thus it suffices to show that the orthogonal complement of $\cW_T$ contains only the origin. Let $\phi\in L^2(M(T,\cO);E)$ satisfy $\langle Wf(T,\cdot),\phi \rangle_{L^2(M;E)}=0$ for all $f\in \cF(T,\cO)$. Let $u\in \CI(\bbR\times M)$ solve
\[ \begin{cases}
(\pa_t^2-\Delta_E)u = 0, & (0,T)\times M \\ u|_{t=T}=0 \; \pa_t u|_{t=T}=\phi
\end{cases}  \]
From Green's identities we have
\[ \langle f,u\rangle_{L^2((0,T)\times M;E)} = \langle (\pa_t^2-\Delta_E)Wf,u\rangle_{L^2((0,T)\times M;E)} - \langle Wf,(\pa_t^2-\Delta_E)u\rangle_{L^2((0,T)\times M;E)} =0   \]
thus $u\equiv 0$ on $(0,T]\times \cO$ by density of $\cF(T,\cO)\subset L^2((0,T)\times \cO;E)$. Extending $u$ across $t=T$ by the odd reflection $ u(t,x)=-u(2T-t,x)$, and denoting the extension by $U$ we have that it satisfies
\[ \begin{cases}
(\pa_t^2-\Delta_E)U=0 & (0,2T)\times M \\ U|_{t=T}=0, \; \pa_tU|_{t=T}=\pa_tu|_{t=T}=\phi
\end{cases}  \]
by our odd reflection, thus $U|_{(0,2T)\times \cO}\equiv 0$. Now by theorem \ref{uniqcont} we conclude that $U|_{C(T,\cO)}\equiv 0$, in particular since $\{T\}\times M(T,\cO)\subset C(T,\cO)$ we have $\phi|_{M(T,\cO)}=\pa_tU|_{\{T\}\times M(T,\cO)}\equiv 0$ as claimed.
\end{proof}

Having proven this fact, the proofs of 1)\; 2) and 3) from \cite{HLOS} generalize immediately to the bundle-valued case. One lemma they depend on is the Blagovestchenskii identity:

\begin{lemma}[Blagovestchenskii Identity]
Let $(M,g)$ be complete. Let $T>0$ and $\cO\subset M$ open and bounded. Let $f,h\in \cF(2T,\cO)$, then
\[ \langle Wf(T,\cdot),Wg(T,\cdot) \rangle_{L^2(M; E)} = \langle f,(JL_{\sD,\cO}^\text{wave}-(L_{\sD,\cO}^\text{wave})^*J)g\rangle_{L^2((0,T)\times M; E)}  \]
where $J:L^2((0,2T);E)\to L^2((0,T);E)$ is the time averaging operator $J\phi(t)=\tfrac{1}{2}\int_t^{2T-t}\phi(s)ds$.
\end{lemma}
%
%
%
%
Another such lemma is that the source-to-source solution operator determines the distance function:
\begin{lemma}
Let $(M,g)$ and $\cO\subset M$ be open and bounded. Then $(\cO, L_{\sD,\cO}^{\text{wave}})$ determines the distance function $d_g$ on $\cO\times \cO$.
\end{lemma}
\begin{proof}
Set $x,y\in \cO$, and choose an auxiliary metric $d_0$ on $\cO$ which induces the same metric space topology on $\cO$ as $g$. Let $\eps>0$ and consider $\cB_{x,\eps}:=(0,\infty)\times B_\eps^{d_0}(x)$ and 
\[ t_\eps:=\inf \{t>0 : \exists f\in \CI_0(\cB_{x,\eps};E), \; \supp{(L_{\sD,\cO}^{\text{wave}}f)}(t,\cdot)\cap B_\eps^{d_0}(y)\neq \emptyset\}.  \]
From finite propagation speed and lemma \eqref{approx-contr} we have 
\[ t_\eps = \text{dist}_g(B_\eps^{d_0}(x),B_\eps^{d_0}(y))   \]
and thus $\lim_{\eps\to 0}t_\eps=d_g(x,y)$.
\end{proof}

Let $\gamma_{y,\xi}(t)$ be the time-$t$ flow of the unique unit speed geodesic starting at $y$ with initial velocity vector $\xi$. Several facts about geometric determination in our setting will rely on the notion of cut time,
\[ \tau(y,\xi)=\sup \{t>0: d_g(y,\gamma_{y,\xi}(t))=t\}. \] 
Note that we have
\begin{equation}\label{exp-closure}
    M = \{ \gamma_{y,\xi}(t)\in M : y\in \cO,\; \xi\in S_yM,\; t<\tau(y,\xi) \}  
\end{equation}  
for $\cO\subset M$ open and non-empty. In other words any point in $M$ can be reached by geodesics originating in $\cO$ which do not meet the cut locus.

Further, the cut distance and distance between points can be determined using only the source-to-source solution data
\[(\cO, g|_{\cO}, h_E|_\cO, L_{\sD,\cO}^{\text{wave}})\]
namely as shown by \cite{HLOS} in the scalar case,
\begin{proposition}
For any $y\in \cO$, $\xi\in S_yM$ we can find the cut time $\tau(y,\xi)$ from the source-to-source solution data $(\cO, g|_{\cO}, h_E|_{\cO}, L_{\sD,\cO}^{\text{wave}})$.

Further, given $z,y\in \cO$, $\xi\in T_y\cO$, $||\eta||=1$ and $r<\tau(y,\eta)$. Then $(\cO, g|_{\cO}, h_E|_{\cO}, L_{\sD,\cO}^{\text{wave}})$ determines $d_g(p,z)$ where $\gamma_{y,\xi}(r)=z$.
\end{proposition}
\begin{proof}
For $y,\xi$ fixed the geodesic segment $\gamma_{y,\xi}([0,s])$ is determined by the given data for $s$ small. Choosing $s>0$ sufficiently small that $\gamma_{y,\xi}([0,s])\subset \cO$, we observe that the condition
\begin{equation}\label{ball-cond}
    \text{there exists $\eps>0$ such that $B_{r+\eps}(x)\subset\overline{B_{s+r}(y)}$}
\end{equation}
determines the cut distance by the forumla,
\[ \tau(y,\xi)=\inf\{ s+r>0: r,s>0, \gamma_{y,\xi}([0,s])\subset\cO, \text{ \eqref{ball-cond} holds} \}.  \]
This follows from its contrapositive, arguing via geodesic continuation of the arc $\gamma_{y,\xi}([0,s])$ and the triangle inequality. So $\tau(y,\xi)$ is determined by relation \eqref{ball-cond}. 

Define $S_\eps(x,r)=(T-(r-\eps),T)\times B_\eps(x)$. Using finite propagation speed and lemma \eqref{approx-contr} we can show that for $p,y,z\in M$, $\eps>0$, and $\ell_p,\ell_y,\ell_z>\eps$ the condition that 
\begin{equation}\label{3ell}
B_{\ell_p}(p)\subset \overline{B_{\ell_y}(y)\cup B_{\ell_z}(z)}  \end{equation}
is equivalent to the existence of $\{f_j\}\subset \CI_0(S_\eps(y,\ell_y)\cup S_\eps(z,\ell_z); E)$ for every $f\in \CI_0(S_\eps(p,\ell_p);E)$ such that
\begin{equation}\label{ball-cond-limit}
     ||Wf(T,\cdot)-Wf_j(T,\cdot)||_{L^2(M;E)}\xrightarrow{j\to \infty} 0.
\end{equation}
Now choosing $\eps>0$ sufficiently small that $B_\eps(y)\cup B_\eps(z)\subset\cO$, and setting $y=z$, $\ell_y=\ell_z=s+r$, and $\ell_p=r+\eps$, we see that relation \eqref{ball-cond} is equivalent to \eqref{ball-cond-limit}. However, using the Blagovestchenskii identity we see that $(\cO, g|_{\cO}, h_E|_{\cO}, L_{\sD,\cO}^{\text{wave}})$ determines \eqref{ball-cond-limit} and thus the cut time as claimed.

The proof of the second claim follows similarly, by considering $s\in (0,r)$ such that $\gamma_{y,\xi}([0,s])\subset \cO$, and set $\gamma_{y,\xi}(s)=p$. Then we can show that 
\[ d_g(p,z) = \inf_{R>0}\{ \eqref{3ell} \text{holds with $\ell_p=r-s+\eps$, $\ell_y=r$, $\ell_z=R$ for some $R$ and $\eps>0$}  \} \] 
by first observing that $d_g(p,z)$ is less than or equal to this infimum (since $r<\tau(y,\xi)$), and then showing this infimum is attained by a similar argument as above, using the Blagovestchenskii identity and lemma \ref{approx-contr}. From this formulation of $d_g(p,z)$ we have that it is determined by the source-to-solution data.
\end{proof}

Using this last proposition and the fact of $\eqref{exp-closure}$, we have that $(\cO, g|_{\cO}, h_E|_{\cO}, L_{\sD,\cO}^{\text{wave}})$ determines the family of distance functions,
\[ R(M):=\{ d_g(x,\cdot)|_{\cO}: x\in M \}\subset \cC^0(\overline{\cO}). \]
In Helin Lassas Oksanen Saksala they prove that this set can be topologized as a smooth Riemannian manifold, isometric to $(M,g)$, and serving as a background space depending only on the behavior of solutions in $\cO$, they show the two Riemannian manifolds are isometric. Their proof, depending only on the distance functions, has no dependence on whether or not the given wave equation is scalar or bundle-valued, and thus we obtain:
\begin{proposition}

Let $\sD_j:=(M_j, E_j, g_j, h_{E_j}, \nabla^{E_j})$ for $j=1,2$ be two Dirac bundles over closed Riemannian manifolds of dimension $n\geq 2$, and let $\cO_j\subset M_j$ be non-empty open sets. Assuming there exists a diffeomorphism $\psi:\cO_1\to\cO_2$ satisfying 
\[ L_{\sD_1,\cO_1}^{\text{wave}}(\psi^*f) =\psi^*(L_{\sD_2,\cO_2}^{\text{wave}}f), \quad \forall f\in \CI_0((0,\infty)\times \cO_2;E).  \]
Then $(M_1,g_1)$ and $(M_2,g_2)$ are isometric Riemannian manifolds. 
\end{proposition}

Having shown the two Riemannian manifolds are isometric, it remains to show that we can also recover the Hermitian bundle structure and connection.
\begin{lemma}\label{ip-det}
Let $f,h\in \cF(2T,\cO)$, for $T>0,\; \cO\subset M$ open. Then $L_{\sD,\cO}^\text{wave}$ determines the inner products
\[ \langle Wf(T,\cdot), Wg(T,\cdot) \rangle_{L^2(M;E)}.  \]
Further, for any $x\in M(T,\cO)$ there exists functions 
\[\{g_\ell\}_{\ell=1}^{\emph{Rk}\; E}\subset \cF(2T,\cO)\]
such that $\{Wg_\ell(T,x))\}_{\ell=1}^{\emph{Rk}\;E}$ forms an orthonormal basis of $E_x$.
\end{lemma}
\begin{proof}
The first claim follows immediately from the Blagovestchenskii identity above. For the second claim we note that it suffices to prove $E_x$ is spanned by the vectors $Wg(T,x)$, for $g\in \cF(T,\cO)$. Thus we are left to show that if $e\in E_x$ and 
\[ \langle e, Wg(T,x)\rangle_{h_E}=0, \quad g\in \cF(T,\cO)  \]
then $e=0$. By duality, applying $W^*: H^{-s}(M;E)\to H^{-s'}((0,T)\times M; E)$ to $e\delta_x$ gives a solution to the initial value problem
\[ \begin{cases}
(\pa_t^2-\Delta_E)u=0 \\ u|_{t=T}=0,\; \pa_tu|_{t=T}=e\delta_x
\end{cases} ,  \]
i.e. $W^*(e\delta_x)=u$. Hence we have $t\mapsto\langle u(t,x), g(t,x)\rangle_{h_E}=0$ for all $t\in (0,T)$. Combining this with the initial condition on $u$ we see $u|_{(0,T]\times M}\equiv 0$. Extending $u$ by the odd reflection across $t=T$ we obtain a section now satisfying $(\pa_t^2-\Delta_E)u=0$ on $(0,2T)\times M$, thus by \ref{uniqcont} we have $e=0$ as claimed.
\end{proof}
In particular this implies that $h_E$ restricted to $E|_{M(T,\cO)}$ for $T>0$ sufficiently small is determined by $L_{\sD,\cO}^{\text{wave}}$.

Next we show how to construct sections of $E$ out of double sequences of functions supported in $(T-r,T)\times \cO$. The following lemma is adapted from \cite{KOP}:
\begin{lemma}\label{double-crit}
Let $\cO\subset M$ be open and let $x\in M$. Then there exists $y\in \cO, \xi\in S_yM$ such that $\gamma_{y,\xi}(t)=x$ for $t<\tau(y,\xi)$. We define $t_k=t+\tfrac{1}{k}$, and set
\[ Y_k=\cO\cap B_{1/k}(y), \quad X_k=M_{\theta<1/k}(Y_k,t_k) \setminus M(\cO,t). \]
where 
\[ M_{\theta<1/k}(Y_k,t_k) := \{ z\in M(Y_k,t_k): \exists \wt y\in Y_k, d_g(z,\wt y)<t-\tfrac{1}{k}, \; \exists \eta\in T_yM, \gamma_{y,\eta}(\tfrac{1}{k})= y, \; \theta(\xi,\eta)<\tfrac{1}{k} \}  \]
with $\theta(\cdot,\cdot)$ equal to the angle between two vectors in $T_yM$ with respect to the inner product induced by $g$. Suppose $\Phi^x=\{f_{jk}\}_{j,k=1}^\infty\subset \CI((T-t_k,T)\times Y_k; E)$ satisfies
\begin{enumerate}
    \item For each $k=1,2,\ldots$ the sequence $\{Wf_{jk}(T,\cdot)\}$ converges weakly in $L^2(M;E)$ to function supported in $X_k$.
    \item There exists $C>0$ such that
    \[ ||Wf_{jk}(T,\cdot)||_{L^2(M;E)}\leq \frac{C}{\emph{Vol}(X_k)^{1/2}}  \]
    \item The limit $\lim_{j,k\to \infty} \langle Wf_{jk}(T,\cdot),Wg(T,\cdot)\rangle_{L^2(M;E)}$ exists for every $g\in \cF(2T,\cO)$.
\end{enumerate}
Then there is a vector $e(x;\Phi^x)\in E_x$ such that
\[ \lim_{j,k\to \infty} \langle Wf_{jk}(T,\cdot),\phi\rangle_{L^2(M;E)} = \langle e(x;\Phi^x), \phi(x)\rangle_{h_E}, \quad \phi\in \CI(M;E) \]
\end{lemma}
This lemma is a consequence of lemma \eqref{ip-det}, and otherwise the proof is the same as given in \cite{KOP}, using the fact that from its construction $\text{diam}(X_k)\to 0$ and $X_k\to x$. Given this criterion for constructing vectors of $e\in E_x$ we observe that such a double sequence $\Phi^x=\{f_{jk}\}_{j,k=1}^\infty$ satisfying the hypotheses of lemma \eqref{double-crit} exists by applying lemma \eqref{approx-contr} to construct a sequence for each $k=1,2,\ldots$
\[  Wf_{jk} \xrightarrow{j\to \infty} \frac{1_{X_k}e}{\Vol{X_k}}\]
where $1_{X_k}$ is the indicator function of $X_k$. This sequence clearly satisfies the conclusions of lemma \eqref{double-crit}, as claimed. Further, if we choose $\Phi^x$ for $x\in U$ such that
\[ x\mapsto \langle e(x;\Phi^x), Wg(T,\cdot)\rangle_{h_E}, \quad g\in \cF(2T,\cO)  \]
is smooth, then the assignment $e(x)= e(x;\Phi^x)$ will be a smooth section of $E|_{U}$ as this implies $e$ depends smoothly in $x$ with respect to an orthonormal frame $\{Wg_\ell(T,\cdot)\}_{\ell=1}^{\text{Rk} E}$ by lemma \eqref{ip-det}.

\begin{theorem}
Let $\cO\subset M$ be open, non-empty, and bounded. Then the Dirac bundle $\sD$ comprised of $(M,g)$, Hermitian vector bundle with connection $(E,h_E, \nabla^E)$ over $M$, and clifford multiplication $\cl:\Cl(T^*M,g)\to \End{E}$ are all determined by the source-to-source solution operator $L_{\sD,\cO}^{\text{wave}}$ and the Hermitian vector bundle $E|_{\cO}$.
\end{theorem}
\begin{proof}
Having already shown the Riemannian manifold $(M,g)$ is determined by $L_{\sD,\cO}^{\text{wave}}$, so it remains to show the given data also determines the Hermitian bundle with connection. By making $\cO$ smaller if needed, we may assume that $\cO\subset M$ is contractible and thus $E|_\cO$ is trivial. 

For each $x\in \cO$ we can choose a double sequence $\Phi^x=\{f_{jk}^x\}_{j,k=1}^\infty$ satisfying the criteria of lemma \ref{double-crit} and hence converges to a prescribed $e\in E_x$. If we further require that this family $(\Phi^x)_{x\in \cO}$ be chosen such that
\[ x\mapsto \langle e(x;\Phi^x), Wg(T,\cdot)\rangle_{h_E}, \quad g\in \cF(2T,\cO) , \]
is smooth, then $e(x)=e(x;\Phi^x)$ will be smooth in $\cO$ as well. 

Beginning with an orthonormal frame $\{e_\ell(x)\}$ for $E|_{\cO}$, we can similarly choose for each of the $\ell=1,\ldots,\text{Rk}\; E$ families of double sequences $\Phi_\ell^x$ representing $e_\ell(x)$. Thus the coefficient functions $h_{i,j}(x)=\langle e_i(x),e_j(x)\rangle_{h_E}=\delta_{ij}(x)$ are determined by $L_{\sD,\cO}^{\text{wave}}$, and in this trivialization $E|_{\cO}$ we can express the Hermitian metric as $\langle v, w\rangle_{h_E}=\langle v^\ell e_\ell(x), w^\kappa e_\kappa(x)\rangle_{h_E}=\sum v^\ell\overline{w^\ell}$.

To determine the bundle and metric globally we need only determine the bundle transition functions. From compactness $\inj_g>c$, and thus there exists a finite cover $\{B_c(p_i)\}_{i=1}^N$ of $M$. Further,
for each $x\in M$ there exists $y\in \cO$ and $\xi\in S_yM$ such that $\gamma_{y,\xi}(t)=x$ with $t<\tau(y,\xi)$. Since each $E|_{B_c(p_i)}\simeq B_c(p_i)\times E_{p_i}$ is trivial, it remains only to determine the transition functions $T^{ij}\in \CI(B_c(p_i)\cap B_c(p_j); \End(E_p))$. With respect to the chosen orthonormal frame $e_\ell(x)$ over $B_c(p_i)\cap B_c(p_j)$ we can write this matrix valued function $T^{ij}(x)=(T_{\ell,\kappa}^{ij}(x))_{\ell,\kappa=1}^{\text{Rk} E}$ in terms of the frame, i.e. by using the local expression of the metric $h_E$ over $B_c(p_i)\cap B_c(p_j)$. The frame $e_\ell(x)$, being generated by applying lemma \ref{double-crit} to points $x\in B_c(p_i)\cap B_c(p_j)$, is thus determined by $L_{\sD,\cO}^{\text{wave}}$ as above. Thus we have recovered the transition functions, hence the entire Hermitian bundle and metric.

Having determined the transition functions, to determine the connection it suffices to determine $\nabla^E$ on a trivialization $E|_{U}\simeq U\times E_p$ where $\nabla^E=d+A$, for $A\in \Omega^1(U;E)$. This portion of the argument follows that of Kurylev Oksanen Paternain and we include it for completeness. Choose $T>0$ sufficiently large that $U\subset M(2T,\cO)$. Since $Wf(T,\cdot)$ is determined by $L_{\sD,\cO}^{\text{wave}}$ for all $f\in \cF(2T,\cO)$, and the wave equation \eqref{direct-wave-eqn} is time-translation invariant, we have also determined $Wf(t,\cdot)$ for $t\in (0,T)$. Differentiating twice in $t$ gives $\Delta_E Wf(t,\cdot)$ for $t\in (0,T)$. Hence by duality
\[ \langle f(T,\cdot), \phi\rangle_{L^2(M;E)} = \langle Wf(T,\cdot), \Delta_E \phi \rangle_{L^2(M;E)}  \quad \phi\in \CI_0(U;E) ,  \]
thus by lemma \eqref{approx-contr} we can determine and densely approximate $L^2(U;E)$ by the functions $Wf(T,\cdot)$ thus determining $\Delta_E\phi$. Now, because $\Delta_E$ differs from $(\nabla^E)^*\nabla^E$ by a zeroth order term, and $\nabla^E=d+A$ over $U$, we compute
\[ (\nabla^E)^*\nabla^E\phi = d^*d\phi -2(A,d\phi)+(d^*A)u -(A,Au),    \]
and conclude that the principal part of $(\nabla^E)^*\nabla^E$ is $d^*d$ and its first order part is $2(A,d\phi)=2g^{jk}A_j\pa_k\phi dx^k$. Because the principal part depends only on the metric $g$, we can determine $d^*d$ in $U$. For $x\in U$ we consider $\phi(x)=\phi_\ell^k(x)$ such that $\phi(x)=0$ and $\pa_j\phi(x)=\delta_{jk}e_\ell(x)$. Thus we can determine the endomorphism piece $A$ of the connection from the first order term in $(\nabla^E)^*\nabla^E\phi(x)$, i.e. determine $A\in \Omega^1(U;E)$ from the determination of 
\[  ((\nabla^E)^*\nabla^E-d^*d)\phi = -2(A,d\phi) = -2g^{ik}(x)A_i e_\ell(x)  \]
for $x\in U$.

Finally, having determined every element of the Dirac bundle $\cD$ except for the Clifford multiplication homomorphism $\cl: \Cl(T^*M,g)\to \End(E)$ we observe that having determined both $\nabla^E$ and $\nabla^g$ we have fixed $\cl$ due to the Clifford compatibility condition that
\[ [\nabla^E,\cl(\theta)] = \cl(\nabla^g\theta), \quad \theta\in T^*M . \]
\end{proof}

\begin{proof}[Proof of Corollary \ref{cor}]
We follow the proof of \cite[Cor 2]{HLOS}. To obtain the total spectral data on $\cO$ we notice that not only are the remaining negative eigenvalues simply minus of the positive eigenvalues, but the unknown negative eigensections are simply the image of the positive eigensections under the \emph{Chirality operator}, 
\[ \gamma = i^{\lceil\frac{m}{2}\rceil} \cl(e_1)\cl(e_2)\ldots\cl(e_m) \in \End(E_x),   \]
where $\{e_j\}_{j=1}^m$ is any choice of orthonormal basis of $T_x\cO$. Thus, because $\D\varphi=\lambda\varphi$ implies $\D(\gamma\varphi)=-\lambda\;\gamma\varphi$ we now have knowledge of all eigensections restricted to $\cO$. Writing the $\D^2=\Delta_E$ eigensections as $\psi_k(x)=\varphi_{k}(x)-\gamma\circ\varphi_{k}(x)$, we now have $\{\lambda_k^2,\psi_k(x) \}_{j=1}^\infty$ as a spectral resolution of $\Delta_E$ on $L^2(M;E)$. From here, as in \cite{HLOS}, given any $f\in \CI_0(\Rp\times \cO)$, we can calculate the approximate `Fourier coefficients' of the wave source-to-solution operator $\cI_k(t)=\langle w^f(t,\cdot),\psi_k\rangle_{L^2}$, and from there recover the approximate source-to-solution operator $\sum\limits_{k=1}^\infty \wt \cI_k(t)\varphi_k(x)=W[h\cdot f](t,x)$ (i.e. the solution associated to a conformal change of source) with $\wt \cI_k(t)=\int_0^t\int_\cO \frac{\sin(\lambda_k t)}{\sin(\lambda_k)}f(s,x)\varphi_j(x)\; h\cdot dV_g(x)ds$. From here the result follows from \ref{main-thm} (after determining the conformal factor $h$, as in \cite[Thm 6]{HLOS}).
\end{proof}

\section{acknowledgements}
Travel funding in support of this work was provided to HQ by the Milliman committee of UW math. The research of GU was partly supported by NSF, a Simons Fellowship, a Walker Family Endowed Professorship at UW and a Si-Yuan Professorship at IAS, HKUST. 

\bibliographystyle{abbrv} 
\bibliography{refs} 

\end{document}